\documentclass[leqno,10pt]{amsart}

\usepackage{amsfonts}
\usepackage{amsmath}
\usepackage{amsfonts}
\usepackage{amssymb}
\usepackage{amscd}
\usepackage{hyperref}
\newtheorem{theorem}{Theorem}
\newtheorem*{acknowledgement*}{Acknowledgement}

\newtheorem*{example*}{Example}

\newtheorem{lemma}[theorem]{Lemma}

\newtheorem{proposition}[theorem]{Proposition}
\newtheorem{remark}[theorem]{Remark}

\def\RRR{{\mathrm R}}

\def\Ric{{{Ric}}}

\begin{document}
\title[On a classification theorem for self--shrinkers]{On a classification theorem for self--shrinkers}

\begin{abstract}
We generalize a classification result for self--shrinkers of the mean curvature flow with nonnegative mean curvature, which was obtained in \cite{CoMi}, replacing the assumption on polynomial volume growth with a weighted $L^2$ condition on the norm of the second fundamental form. Our approach adopt the viewpoint of weighted manifolds and permits also to recover and to extend some others recent classification and gap results for self--shrinkers. 
\end{abstract}

\date{\today}

\author{Michele Rimoldi}
\address{Dipartimento di Scienza e Alta Tecnologia\\
Universit\`a degli Studi dell'Insubria\\
via Valleggio 11\\
I-22100 Como, ITALY} 
\email{michele.rimoldi@gmail.com}%


\maketitle


\section{Introduction}\label{Intro}
Let $M^m$ be a complete $m$--dimensional Riemannian manifold without boundary smoothly immersed by $x_0:M^m\to\mathbb{R}^{m+1}$ as a hypersurface in the Euclidean space $\mathbb{R}^{m+1}$. We say that $M_0=x_0(M)$ is moved along its mean curvature vector if there is a whole family $x(\cdot\,,\, t)$ of smooth immersions, with corresponding hypersurfaces $x_t=x(\cdot\,,\,t)(M)$, such that it satisfies the mean curvature flow initial value problem
\begin{equation}\label{MCF}
\begin{cases}
\frac{\partial}{\partial t}x(p,t)=-H(p,t)\nu(p,t)&p\in M^m\\
x(\cdot,\,0)=x_0.&
\end{cases} 
\end{equation}
Here $H(p,t)$ and $\nu(p,t)$ are respectively the mean curvature and the outer unit normal vector of the hypersurface $M_t$ at $x(p,t)$. The short time existence and uniqueness of a solution of \eqref{MCF} was investigated in classical works on quasilinear parabolic equations.
 
We are interested in the study of self--shrinking solutions of the flow \eqref{MCF}. A MCF $M_t$ is called a self--shrinking solution if it satifies
\[
\ M_t=\sqrt{-2t}M_{-\frac{1}{2}}
\]
For an overview on the role that such solutions play in the study of MCF see e.g. the introduction in \cite{CoMi}. A hypersurface is said to be a self--shrinker if it is the time $t=-\frac{1}{2}$ slice of a self--shrinking MCF. By \cite[Lemma 2.2]{CoMi} we will simply think of a self-shrinker as a hypersurface $x: M^m\to\mathbb{R}^{m+1}$ satisfying the following equation for the mean curvature $H$ and the (outer) unit normal $\nu$
\begin{equation}\label{SS}
H=\left\langle x,\nu\right\rangle.
\end{equation}

\emph{In this note we will consider only the codimension one case.}

U. Abresch and J. Langer, \cite{AL}, completely classified immersed self--shrinkers of dimension one. These locally mean convex curves are now called Abresch--Langer curves. Moreover they also showed that the circle is the only simple closed self--shrinking curve. This was later generalized to higher dimension by G. Huisken, \cite{Hu1}, who showed that the only smooth closed self--shrinkers $M^m\to\mathbb{R}^{m+1}$ with nonnegative mean curvature are round spheres.

In the complete non--compact case, G. Huisken, \cite{Hu3}, also proved a classification theorem saying that the only possible smooth self--shrinkers $M^m\to\mathbb{R}^{m+1}$ with nonnegative mean curvature, bounded norm of the second fundamental form $A$ and polynomial volume growth are isometric to $\Gamma\times\mathbb{R}^{m-1}$ or $\mathbb{S}^k(\sqrt{k})\times\mathbb{R}^{m-k}$ ($0\leq k\leq m$), where $\Gamma$ is an Abresch--Langer curve. Asking also embeddedness one can rule out the product of immersed Abresch--Langer curves with Euclidean factors. 

Recently Colding and Minicozzi, \cite[Section 10]{CoMi} showed that the hypothesis $|A|$ bounded can be dropped. Indeed they obtain the following
\begin{theorem}\label{ThCoMi}(Theorem 0.17 in \cite{CoMi}). 
$\mathbb{S}^k(\sqrt{k})\times\mathbb{R}^{m-k}$ $(0\leq k\leq m)$ are the only smooth complete embedded self--shrinkers without boundary, with polynomial volume growth, and $H\geq 0$ in $\mathbb{R}^{m+1}$.
\end{theorem}

The hypothesis of polynomial volume growth in \cite{CoMi} is used to show that various weighted integrals converge in order to justify various integration by parts. This assumption is natural in the study of the singularities that a MCF goes through since any time--slice of a blow up of a closed MCF has polynomial volume growth (see e.g. Corollary 2.13 in \cite{CoMi}). Nevertheless, looking only at the self--shrinker equation \eqref{SS}, it might be thought under what weaker conditions the conclusion in Theorem \ref{ThCoMi} still holds. In the main result of this note we replace polynomial volume growth with a weighted $L^2$--condition on the norm of the second fundamental form. Note that, in case $H>0$, the polynomial volume growth assumption impies that $|A|\in L^2(e^{-\frac{|x|^2}{2}}d\rm{vol})$ (see Proposition 10.14 in \cite{CoMi}). 
\begin{theorem}\label{LiouvilleAppl}
$\mathbb{S}^k(\sqrt{k})\times\mathbb{R}^{m-k}$ $(0\leq k\leq m)$ are the only smooth complete embedded self--shrinkers without boundary in $\mathbb{R}^{m+1}$ with $H\geq 0$ and $|A|\in L^2( e^{-\frac{|x|^2}{2}}d\rm{vol})$.
\end{theorem} 

The viewpoint of weighted manifold adopted in the proof of Theorem \ref{LiouvilleAppl} permit also to recover easily a classification result by H. D. Cao and H. Li, \cite{CL} (actually they prove this result in arbitrary codimension). Unlike Theorem \ref{LiouvilleAppl}, in this result it is considered a $L^\infty$--type condition on the norm of the second fundamental form and the mean curvature is no longer assumed to be nonnegative.
\begin{theorem}\label{ThCaoLi}(Theorem 1.1 in \cite{CL})
If $M^m\to\mathbb{R}^{m+1}$ is an $m$--dimensional complete self--shrinker without boundary and with polynomial volume growth, and $|A|^2\leq 1$ then $M$ is of the form $\mathbb{S}^k(\sqrt{k})\times\mathbb{R}^{m-k}$ for some $0\leq k\leq m$.
\end{theorem}

As a side product of our techniques, we shall provide a straightforward proof of this result.

Moreover it was observed in \cite{CL} that, as a direct consequence of Theorem \ref{ThCaoLi}, one can formulate a gap result for $|A|$ saying that if a self--shrinker has polynomial volume growth and $|A|^2<1$ then $M$ is a hyperplane in $\mathbb{R}^{m+1}$, thus recovering a result in \cite{LeSe}. By making use of the validity of the full Omori--Yau maximum principle for the suitable weighted Laplacian on a complete self--shrinker when $|A|$ is bounded, Q.--M. Cheng and Y. Peng, \cite{CP} obtained the same conclusion without assuming the polynomial volume growth but asking the stronger condition
$\sup|A|^2<1$ to hold.
This result can be improved with the following
\begin{proposition}\label{ImpChengPeng}
Let $M^m\to\mathbb{R}^{m+1}$ be a complete self--shrinker without boundary. Denote by $r(x)$ the geodesic distance from a fixed origin and assume that
\begin{equation}\label{Contr4f-par}
|A|^2\leq 1- C (1+r(x))^{-\mu}
\end{equation}
for some constants $0<C\leq1$ and $0\leq\mu\leq1$. Then $M$ is a hyperplane in $\mathbb{R}^{m+1}$.
\end{proposition}

\section{Basic equations}
We will use the notation of \cite{Hu1, Hu2}.

{\em In all of this note the Einstein convention of summing over repeated indices from $1$ to $m$ will be adopted.}

Let $x:M^m\to\mathbb{R}^{m+1}$ be a generic hypersurface smoothly immersed in $\mathbb{R}^{m+1}$ and let us consider an orthonormal frame $\left\{e_i\right\}$ of $M$.
The coefficients of the second fundamental form $A$ are defined to be
\begin{equation}\label{SecFundForm}
h_{ij}=-\left\langle \nabla_ie_j,\nu\right\rangle.
\end{equation}
In particular we have 
\begin{equation}\label{SecFundForm2} 
\nabla_i\nu=h_{ij}e_j. 
\end{equation}
Since $\left\langle \nabla_\nu\nu,\nu\right\rangle=0$ we have that the mean curvature $H=\mathrm{div} \nu$ is given by $$H=\left\langle \nabla_i\nu, e_i\right\rangle=h_{ii}.$$ The mean curvature vector is defined to be $\mathbf{H}=-H\nu$.

The Riemann curvature tensor, the Ricci tensor and the scalar curvature are given by Gauss' equation
\begin{eqnarray}
\RRR_{ijkl}&=&h_{ik}h_{jl}-h_{il}h_{jk}\label{Riem}\\
\RRR_{ik}&=&Hh_{ik}-h_{il}h_{lk}\label{Ric}\\
\RRR&=&H^2-|A|^2.\label{R}
\end{eqnarray}
Using Codazzi's equations
\begin{equation}\label{Codazzi}
\nabla_ih_{kl}=\nabla_kh_{il}=\nabla_{l}h_{ik},
\end{equation}
and commutation formulas for the interchange of two covariant derivatives combined with \eqref{Riem}, we obtain the validity of the following well-known identity
\begin{equation}
\Delta h_{ij}=\nabla_i\nabla_j H+Hh_{il}h_{lj}-|A|^2
h_{ij},\label{Sim1}
\end{equation}
which, contracted with $h_{ij}$, gives Simon's identity
\begin{equation}\label{Sim2}
\frac{1}{2}\Delta|A|^2= h_{ij}\nabla_i\nabla_j H+|\nabla A|^2+2H\mathrm{tr}(A^3)-2|A|^4.
\end{equation}
We will also need the following well--known inequality 
\begin{equation}\label{KatoA}
|\nabla|A||^2\leq|\nabla A|^2.
\end{equation}
Finally, let us denote by $\mathcal{L}$ the linear elliptic operator defined on $v\in C^{\infty}(M)$ by
\begin{equation}\label{Operator}
\mathcal{L}v=\Delta v-\left\langle x, \nabla v\right\rangle=e^{\frac{|x|^2}{2}}\mathrm{div}(e^{-\frac{|x|^2}{2}}\nabla v),
\end{equation}
where $\Delta$, $\nabla$, $\mathrm{div}$ are respectively the Laplacian, the gradient, and the divergence on the hypersurface $M$. The operator $\mathcal{L}$ is clearly a symmetric operator on $L^2(M, e^{-\frac{|x|^2}{2}}d\rm{vol})$ and plays a very important role (also) in the study of self-shrinkers; for more details see \cite{CoMi}. Note that using well-known notations we have that \[
\mathcal{L}=\Delta_{x^T}=\Delta_{\frac{|x|^2}{2}},
\]
where $x^T$ is the tangential projection of $x$.
In the next lemma we collect two basic identities for self--shrinkers, which can be found in the proof of Theorem 4.1 in \cite{Hu2}.
\begin{lemma}\label{Equations}
Let $x:M^m\to\mathbb{R}^{m+1}$ be a smoothly immersed self-shrinker in $\mathbb{R}^{m+1}$. then the following identities hold:
\begin{eqnarray}
\mathcal{L}H&=&H(1-|A|^2)\label{Id1}\\
\mathcal{L}|A|^2&=&2|\nabla A|^2+2|A|^2-2|A|^4\label{Id2}
\end{eqnarray}
\end{lemma}
\begin{proof}
Differentiating \eqref{SS} and using \eqref{SecFundForm2} and the fact that $\nabla_i x=e_i$ for every $i=1,\ldots, n$, we obtain
\[
\ \nabla_i H= \left\langle x,e_j\right\rangle h_{ij}.
\]
working at a point $p$ and choosing the frame $e_i$ so that $\nabla^{T}_ie_j(p)=0$, differentiating again gives at $p$ that
\begin{equation}\label{HessH}
\nabla_i\nabla_j H=h_{ij}-Hh_{il}h_{jl}+\left\langle x, e_l\right\rangle\nabla_l h_{ij}.
\end{equation}
Tracing, we thus get
\[
\ \Delta H= H-H|A|^2+\left\langle x, e_l\right\rangle\nabla_l H,
\]
that is, \eqref{Id1}.

On the other hand, contracting \eqref{HessH} with $h_{ij}$, we obtain
\begin{equation}\label{hijHessH}
h_{ij}\nabla_i\nabla_jH=|A|^2-H\mathrm{tr}(A^3)+\left\langle x, e_l\right\rangle\nabla_lh_{ij}h_{ij}.
\end{equation}
By \eqref{Sim2} and \eqref{hijHessH} we get
\[
\ \Delta|A|^2=2|\nabla A|^2+2|A|^2-2|A|^4+\left\langle x, \nabla_l |A|^2e_l\right\rangle,
\]
that is, \eqref{Id2}.
\end{proof}
Combining \eqref{Id2} and \eqref{KatoA} the Simon's inequality for $|A|$ on a self--shrinker reads as follows. 
\begin{lemma}\label{SimonLem}(Lemma 10.8 in \cite{CoMi}).
Let $x:M^m\to\mathbb{R}^{m+1}$ be a smoothly immersed self--shrinker in $\mathbb{R}^{m+1}$ then  
\begin{equation}\label{Simon1}
|A|\left[\mathcal{L}+(|A|^2-1)\right]|A|=|\nabla A|^2-|\nabla|A||^2\geq0.
\end{equation}
\end{lemma}
We shall also use the next computations concerning the square norm of the immersion of a self--shrinker; \cite[Lemma 3.20]{CoMi}.
\begin{lemma}
Let $x:M^m\to\mathbb{R}^{m+1}$ be a smoothly immersed self--shrinker in $\mathbb{R}^{m+1}$ then
\begin{eqnarray}
\mathcal{L}|x|^2&=&2(m-|x|^2)\label{MorFor1}\\
\Delta|x|^2&=&2(m-H^2).\label{MorFor2}
\end{eqnarray}
\end{lemma}

\section{Self-shrinkers as weighted manifolds}
Looking at the  basic formulas we have presented in the previous section we are naturally led to think of a self--shrinker $x:M^m\to\mathbb{R}^{m+1}$ as a weighted manifold $(M^m, e^{-\frac{|x|^{2}}{2}}d\rm{vol})$, whose geometry is governed by analytic properties of the operator $\mathcal{L}$.

We look now at the associate $\infty$--Bakry--Emery Ricci tensor. Letting $f=\frac{|x|^2}{2}$, working at a point $p$, and choosing the frame in such a way that $\nabla^{T}_ie_j(p)=0$, one gets that at $p$
\begin{eqnarray*}
\RRR_{ij}+\nabla_i\nabla_jf&=&\RRR_{ij}+\left\langle \nabla_i x^{T}, e_j\right\rangle+\left\langle x^{T},\nabla_i e_j\right\rangle\\
&=&\RRR_{ij}+\left\langle \nabla_i x, e_j\right\rangle-\left\langle \nabla_i(\left\langle x, \nu\right\rangle\nu), e_j\right\rangle+\left\langle x, \nabla_i e_j\right\rangle-\left\langle x, \nu\right\rangle\left\langle \nu, \nabla_i e_j\right\rangle\\
&=&\RRR_{ij}+\delta_{ij}-\left\langle x, \nu\right\rangle\left\langle \nabla_i\nu, e_j\right\rangle+\left\langle x, \left\langle \nabla_ie_j, \nu\right\rangle \nu\right\rangle-\left\langle x, \nu\right\rangle\left\langle \nu, \nabla_ie_j\right\rangle\\
&=&\RRR_{ij}+\delta_{ij}-\left\langle x, \nu\right\rangle\left\langle \nabla_i \nu, e_j\right\rangle\\
&=&\RRR_{ij}+\delta_{ij}-\left\langle x, \nu\right\rangle h_{ij}.
\end{eqnarray*}
Hence, by the self--shrinker equation \eqref{SS} we get that
\[
\ \left(\Ric_{f}\right)_{ij}=\RRR_{ij}+\delta_{ij}-Hh_{ij},
\]
that is, using \eqref{Ric},
\[
\ \left(\Ric_{f}\right)_{ij}=+\delta_{ij}-h_{il}h_{lj}.
\]
We have thus obtained the following lower bound for the $\infty$--Bakry--Emery Ricci tensor of a self-shrinker,
\begin{equation}\label{BoundRicf}
\Ric_f\geq-(|A|^2-1).
\end{equation}

\section{A Liouville--type theorem}
Some of the results of the next section rely on the following weighted version of a Liouville-type result originally obtained in \cite{PRS-JFA05}, \cite{PRS-Book}; see also \cite{PV1}.
\begin{theorem}\label{Liouville}
Assume that on a complete weighted manifold $(M, e^{-f}d\rm{vol})$ the locally Lipschitz functions $u\geq 0$, $v>0$ satisfy
\begin{equation}\label{ineq_u}
\Delta_fu+a(x)u\geq0
\end{equation}
and
\begin{equation}\label{ineq_deltav}
\Delta_fv+\delta a(x)v\leq0,
\end{equation}
for some constant $\delta\geq1$ and $a(x)\in C^0(M)$. If $u\in L^{2\beta}\left(M,e^{-f}d\mathrm{vol}\right)$, $1\leq\beta\leq \delta$, then there exists a constant $C\geq0$ such that 
\[
\ u^{\delta}=Cv.
\]
Furthermore,
\begin{enumerate}
	\item [(i)]If $\delta>1$ then $u$ is constant on $M$ and either $a\equiv 0$ or $u\equiv 0$.
	\item [(ii)]If $\delta=1$ and $u\not\equiv0$, $v$ and therefore $u^\delta$ satisfy \eqref{ineq_deltav} with equality sign.
\end{enumerate}
\end{theorem}
We shall use a number of straightforward computations that, for the sake of clarity, we isolate in the following
\begin{lemma}\label{computation}
Suppose we are given $u,v,f:M\rightarrow\mathbb{R}$ with $v>0$. Having fixed $\alpha,\beta>0$ set%
\[
h=-\log v^{2\alpha}+f.
\]
Then, the following identity holds%
\[
\Delta_{h}\left(  \frac{u^{\beta}}{v^{\alpha}}\right)  =\frac{u^{\beta-1}%
}{v^{\alpha+1}}\left\{  \beta v\Delta_{f}u-\alpha u\Delta_{f}v+\beta\left(
\beta-1\right)  v\frac{\left\vert \nabla u\right\vert ^{2}}{u}-\alpha\left(
\alpha-1\right)  u\frac{\left\vert \nabla v\right\vert ^{2}}{v}\right\}  .
\]
Moreover%
\[
\left\Vert \frac{u^{\beta}}{v^{\alpha}}\right\Vert _{L^{2}\left(
M,e^{-h}d\mathrm{vol}\right)  }=\left\Vert u^{\beta}\right\Vert _{L^{2}\left(
M,e^{-f}d\mathrm{vol}\right)  }.
\]

\end{lemma}
\begin{proof}(of Theorem \ref{Liouville})
We follow the arguments in Theorem 1.4 of \cite{PRS-JFA05}, and Theorem 4.5 of \cite{PRS-Book}. Applying Lemma \ref{computation} with $\beta\geq1$ and $\alpha=\frac{\beta}{\delta}$, we have that
\begin{align*}
\Delta_{h}\left(  \frac{u^{\beta}}{v^{\alpha}}\right)  &\geq\frac{u^{\beta-1}%
}{v^{\alpha+1}}\left\{  \beta v\Delta_{f}u-\alpha u\Delta_{f}v\right\}\\
&\geq\frac{u^{\beta-1}}{v^{\alpha+1}}\left\{\alpha \delta auv-\beta auv\right\},
\end{align*}
whence,
\[
\ \Delta_h\left(\frac{u^{\beta}}{v^{\alpha}}\right)\geq 0.
\]
Since, again by Lemma \ref{computation},
\[
\left\Vert \frac{u^{\beta}}{v^{\alpha}}\right\Vert _{L^{2}\left(
M,e^{-h}d\mathrm{vol}\right)  }=\left\Vert u^{\beta}\right\Vert _{L^{2}\left(
M,e^{-f}d\mathrm{vol}\right)  },
\] 
where $h=-\log v^{2\alpha}+f$, applying to $u^{\beta}/v^{\alpha
}$ the $L^{2}$-Liouville theorem for non-negative $\Delta_{h}%
$--subharmonic functions, see \cite{PRS-JFA05}, \cite{PRS-Book}, we can conclude that%
\[
\frac{u^{\beta}}{v^{\alpha}}=\mathrm{const}.
\]
Equivalently, there exist a constant $C\geq0$ such that
\[
\ u^{\delta}=Cv.
\]
If we now assume that $u\not\equiv 0$ multiplying by a suitable constant we may suppose that $u$ is strictly positive, and
\[
\ u^{\delta}=v. 
\]
Inserting this latter into \eqref{ineq_deltav} and using \eqref{ineq_u}, we deduce
\begin{align}\label{ineq_u^delta}
0\geq\Delta_fu^\delta+\delta au^\delta&=\delta u^{\delta-2}((\delta-1)\left|\nabla u\right|^2+u\Delta_f u+u^2a)\\
&\geq \delta(\delta-1)u^{\delta-2}|\nabla u|^{2}.\nonumber
\end{align}
Thus, if $\delta>1$, $|\nabla u|^2\equiv0$ proving that $u$ and therefore $v$ are constant. It follows from \eqref{ineq_deltav} that
\[
\ \Delta_f v+\delta av=\delta av\leq0,
\] 
so that $a\leq 0$, while, from \eqref{ineq_u},
\[
\ 0\leq\Delta_f u+ au=au
\]
and we conclude that $a\equiv0$.\\
Finally, assume that $\delta=1$. Then according to \eqref{ineq_u^delta}
\begin{equation}\label{eq_u^delta}
\Delta_fu^\delta+\delta au^\delta=0
\end{equation}
i.e $v=u^\delta$ satisfies \eqref{ineq_deltav} with equality sign.
\end{proof}

\section{Proof of the classification results}

We are now in the position to prove the main result of the paper.

\begin{proof}(of Theorem \ref{LiouvilleAppl})
Since $H\geq 0$ and it satisfies \eqref{Id1} note that by the minimum principle either $H\equiv0$ or $H>0$. In case $H\equiv 0$, the self--shrinking equation \eqref{SS} implies that $M^m$ is a hyperplane through the origin. Therefore, assume by now that $H>0$.

By equations \eqref{Id1} and \eqref{Simon1} we can apply Theorem \ref{Liouville} with the choiches $u=|A|$, $v=H$, $a(x)=|A|^2-1$ and $\delta=\beta=1$ to deduce that 
\begin{equation}\label{Key1}
|A|=CH,
\end{equation} 
for some constant $C\geq 0$ and that either $|A|\equiv 0$, and $M^m$ is necessarily a hyperplane through the origin, or 
\[
\ |\nabla A|^2=|\nabla|A||^2.
\]
These are the key geometric identities to conlude the classification as in \cite{Hu3}, \cite{CoMi}.
\end{proof}

\begin{remark}
\rm{Note that, even though equation \eqref{Key1} is proved also in \cite{CoMi}, it is also used there the fact that, in case of polynomial volume growth and $H>0$, $|A|^2,\,|\nabla A|, |\nabla|A||\in L^2(e^{-\frac{|x|^2}{2}}d\rm{vol})$. This is a reason of interest in the approach we propose since we clarify what is really needed to get the conclusion in \cite{CoMi}.}
\end{remark}

Recall that a weighted manifold $(M, g, e^{-f}d\rm{vol})$ is said to be $f$--parabolic if every solution of $\Delta_f u\geq 0$ satisfying $u^{*}=\sup_Mu<+\infty$ must be identically constant. It can be shown that a sufficient condition for $(M, g, e^{-f}d\rm{vol})$ to be $f$--parabolic is that $M$ is geodesically complete and
\begin{equation}\label{f-par}
\mathrm{vol}_f(\partial B_r)^{-1}=\left(\int_{\partial B_r}e^{-f}d\mathrm{vol}_{m-1}\right)^{-1}\notin L^{1}(+\infty),
\end{equation}
where $d\mathrm{vol}_{m-1}$ stands for the $(m-1)$--Hausdorff measure.
\begin{remark}\label{Condf-par}
\rm{
By the Qian--Wei--Wylie weighted volume estimates, \cite{Q1, WW}, and Proposition 4.3 in \cite{PRRS} we hence obtain the $f$--parabolicity of weighted manifolds $(M, \left\langle \,,\,\right\rangle, e^{-f}d\rm{vol})$ satisfying one of the following curvature assumptions
\begin{itemize}
	\item [(a)]$Ric_f\geq \epsilon>0$, \quad $\epsilon$ constant;
	\item [(b)]$Ric_f\geq D(1+r)^{-\mu}$\quad with $D>0$ and $0\leq\mu\leq1$.
\end{itemize}
Note that (a) actually implies the stronger condition $\mathrm{vol}_f(M)<+\infty$.
}
\end{remark}
\begin{remark}
\rm{
A properly immersed self--shrinker $x:M^m\to\mathbb{R}^{m+1}$ is $f$--parabolic with $f=\frac{|x|^2}{2}$. Indeed, by \eqref{MorFor1}, $|x|^2$ is a positive, proper function satisfying $\mathcal{L}|x|^2\leq0$ outside a compact set and the standard Khas'minskii criterion of parabolicity applies. The same conclusion could be obtained in a more involved way using a nice result by Q. Ding and Y. L. Xin, \cite{QDi}; see also \cite{ChZo}.

Similarly, using \eqref{MorFor2} and the fact that $|x|\geq |\left\langle x,\nu\right\rangle|=|H|$, we see that a self--shrinker satisfying $|H|(x)\to\infty$ as $x\to\infty$ is parabolic with respect to the Laplace--Beltrami operator. The same conclusion can be reached if we consider a properly immersed self--shrinker such that $|H|\geq\sqrt{m}$ outside a compact set.}
\end{remark}
Using this property we can now give a straightforward proof of the classification result by H. D. Cao and H. Li stated in Theorem \ref{ThCaoLi}.

\begin{proof}(of Theorem \ref{ThCaoLi})
Since $M$ has polynomial volume growth, we immediately obtain that $M$ satisfies $\mathrm{vol}_{\frac{|x|^2}{2}}(M)<+\infty$. In particular we have that $M$ is $\frac{|x|^2}{2}$\,--parabolic. Using the fact that $|A|^2\leq 1$ in \eqref{Simon1} we obtain that $|A|$ is a bounded $\frac{|x|^2}{2}$\,--\,subharmonic function, therefore it has to be constant. Substituing again in \eqref{Simon1} we get $|\nabla A|^2=|\nabla|A||^2$ and hence that $|\nabla A|=0$. According to a theorem of Lawson, \cite{La}, we conclude the desired classification.  
\end{proof}
Similarly, using condition (b) in Remark \ref{Condf-par}, we can prove Proposition \ref{ImpChengPeng}.
\begin{proof}(of Proposition \ref{ImpChengPeng})
By \eqref{BoundRicf} and Remark \ref{Condf-par} we know that $M$ is $\frac{|x|^2}{2}$\,--parabolic and by \eqref{Simon1} we conclude that $|A|$ is constant. Moreover by \eqref{Simon1} and  \eqref{Contr4f-par} we necessarily have $|A|\equiv 0$, that is, $M$ is a hyperplane in $\mathbb{R}^{m+1}$.
\end{proof}

\begin{acknowledgement*}
The author is deeply grateful to Stefano Pigola for useful conversations during the preparation of the manuscript.
\end{acknowledgement*}

\bibliographystyle{amsplain}
\bibliography{ClassResSelfShrink}
\end{document}